\documentclass{amsart}
\usepackage{amsmath,amssymb,amsthm, mathrsfs}
\usepackage{boxedminipage}
\usepackage{graphicx}
\usepackage[margin=1in]{geometry}
\usepackage{mathpazo}
\usepackage{fancyhdr}
\usepackage[utf8]{inputenc}
\usepackage{color}
\usepackage{tikz}

\newtheorem{theorem}{Theorem}[section]
\newtheorem{lemma}[theorem]{Lemma}

\newtheorem{corollary}[theorem]{Corollary}
\newtheorem{proposition}[theorem]{Proposition}

\newtheorem{varexample}[theorem]{Example}

\newenvironment{example}{\begin{varexample}
\begin{normalfont}}{\end{normalfont}
\end{varexample}}

\theoremstyle{definition}

\newtheorem{definition}[theorem]{Definition}

\definecolor{gold}{rgb}{1.0,0.75,0.0}

\title{A New Lower Bound on Graph Gonality}
\author{Michael Harp}
\author{Elijah Jackson}
\author{David Jensen}
\author{Noah Speeter}

\begin{document}

\maketitle

\begin{abstract}
We define a new graph invariant called the scramble number.  We show that the scramble number of a graph is a lower bound for the gonality and an upper bound for the treewidth.  Unlike the treewdith, the scramble number is not minor monotone, but it is subgraph monotone and invariant under subdivision.  We compute the scramble number and gonality of several families of graphs for which these invariants are strictly greater than the treewidth. 
\end{abstract}

\section{Introduction}

In \cite{BakerNorine}, Baker and Norine introduce the theory of divisors on graphs, and in \cite{Baker08}, Baker defines a new graph invariant known as the \emph{gonality}.  Due to its connection to algebraic geometry, this invariant has received a great deal of interest \cite{dbG, NPHard, RandomGonality, ExpanderGonality, Hendrey, AidunMorrison, Morrison}.  Computing the gonality $\mathrm{gon}(G)$ of a graph $G$ is NP-hard \cite{NPHard}.  To find an upper bound, one only has to produce an example of a divisor with positive rank, but it is much more difficult to find lower bounds.  A significant step in this direction was obtained in \cite{dbG}, in which the authors show that the gonality of a graph $G$ is bounded below by a much-studied graph invariant known as the \emph{treewidth}, $\mathrm{tw}(G)$.

In this paper we define a new graph invariant, which we call the \emph{scramble number} of $G$ and denote $\mathrm{sn}(G)$.  We refer the reader to \S~\ref{Sec:Scramble} for a definition.  Our main result is the following.

\begin{theorem}
\label{Thm:MainThm}
For any graph $G$, we have
\[
\mathrm{tw} (G) \leq \mathrm{sn} (G) \leq \mathrm{gon} (G) .
\]
\end{theorem}

Theorem~\ref{Thm:MainThm} is proved in two parts.  The left inequality is proved in Corollary~\ref{Cor:TW} and the right inequality in Theorem~\ref{Thm:Bound}.  The proof of Theorem~\ref{Thm:Bound} follows closely that of \cite{dbG}.  Indeed, the scramble number is defined in such a way as to generalize the statement of \cite[Theorem~3.1]{dbG} without significantly altering its proof. 

After establishing Theorem~\ref{Thm:MainThm}, we then examine properties of the scramble number.  We show that it is subgraph monotone (Proposition~\ref{Prop:Subgraph}) and invariant under subdivision (Proposition~\ref{Prop:Refinement}) but not minor montone (Example~\ref{Ex:Minor}).  In Examples~\ref{Ex:Path} and~\ref{Ex:COL}, we show that $\mathrm{sn} (G)$ is unbounded in $\mathrm{tw} (G)$, and $\mathrm{gon} (G)$ is unbounded in $\mathrm{sn} (G)$.

The fact that the treewidth bound can be strict is not new.  See, for example, \cite{Hendrey}, where it is shown that, for all integers $k \geq 2$ and $\ell \geq k$, there exists a $k$-connected graph of treewidth $k$ and gonality at least $\ell$.  The strength of Theorem~\ref{Thm:MainThm} is that it can be used to compute the gonality in cases where this bound is strict.  This is demonstrated in \S~\ref{Sec:Examples}, where we use Theorem~\ref{Thm:MainThm} to compute the gonality of several families of graphs.  In \cite{Morrison}, the authors compute the treewidth of the \emph{grid graphs} $G_{m,n}$, the \emph{stacked prisms} $Y_{m,n}$, and the \emph{toroidal grid graphs} $T_{m,n}$.  (See \S~\ref{Sec:Examples} for precise definitions of these graphs.)  Combining their results with the bound from \cite{dbG}, they compute the gonalities of all these graphs except for $T_{n,n}$, $T_{n+1,n}$, and $Y_{2n,n}$.  Using Theorem~\ref{Thm:MainThm}, we complete this project, computing the gonalities of the graphs in these families and some minor generalizations.  Even in cases where the gonality was already known, our constructions are quite a bit simpler than those that arise in computations of treewidth.  For this reason, we suspect that scrambles may be useful in further computations of graph gonality.

\textbf{Acknowledgments.}  This research was conducted as a project with the University of Kentucky Math Lab.  The third author was supported by NSF DMS-1601896.  We would like to thank Ralph Morrison for some discussions on this material.  We would also like to thank the anonymous referees for several helpful comments, including the suggestion to add Examples~\ref{Ex:Path} and~\ref{Ex:COL}.  

\section{Preliminiaries}

In this section, we fix notation and review some basic definitions.  For simplicity, we will assume throughout that all graphs are connected, possibly with multiple edges, but no loops.  Given a graph $G$, we write $V(G)$ for its vertex set and $E(G)$ for its edge set.  If $A$ is a subset of $V(G)$, we write $A^c$ for its complement.  If $A$ and $B$ are subsets of $V(G)$, we write $E_G (A,B)$ for the set of edges with one endpoint in $A$ and the other endpoint in $B$.  We write $E (A,B)$ when the graph $G$ is clear from the context.  A set of vertices $B \subseteq V(G)$ is \emph{connected} if, for every proper subset $A \subsetneq B$, the set $E(A, B \smallsetminus A)$ is nonempty.

A \emph{subgraph} of a graph $G$ is a graph that can be obtained from $G$ by deleting edges and deleting isolated vertices.  A \emph{minor} of a graph $G$ is a graph that can be obtained from $G$ by contracting edges, deleting edges, and deleting isolated vertices.  If $e \in E(G)$ is an edge with endpoints $v$ and $w$, the \emph{edge-subdivision} of $G$ at $e$ is obtained by introducing a new vertex $u$ in the middle of $e$.  In other words, the vertex set of the edge-subdivision is $V(G) \cup \{ u \}$, and the edge set of the edge-subdivision is $E(G) \cup \{ uv, uw \} \smallsetminus \{ e \}$.  A \emph{subdivision} of a graph $G$ is one that can be obtained from $G$ by finitely many edge-subdivisions.

We briefly describe the theory of divisors on graphs.  For a more detailed treatment, we refer the reader to \cite{BakerNorine} or \cite{BakerJensen}.  A \emph{divisor} on a graph $G$ is a formal $\mathbb{Z}$-linear combination of vertices of $G$.  It is standard to think of a divisor as a stack of poker chips on each of the vertices, with negative coefficients corresponding to vertices that are ``in debt''.  For this reason, divisors are sometimes referred to as \emph{chip configurations}.  (We use the term ``divisor'' to emphasize the connection with algebraic geometry.)  The \emph{degree} of a divisor is the sum of the coefficients:
\[
\mathrm{deg} \left( \sum_{v \in V(G)} a_v \cdot v \right) := \sum_{v \in V(G)} a_v .
\]
In other words, the degree is the total number of poker chips.  A divisor is said to be \emph{effective} if all of the coefficients are nonnegative -- that is, none of the vertices are in debt.  The \emph{support} of an effective divisor $D$, denoted $\mathrm{Supp}(D)$, is the set of vertices with nonzero coefficient.

Given a divisor $D$ on $G$ and a vertex $v \in V(G)$, we may \emph{fire} the vertex $v$ to obtain a new divisor.  For each edge $e$ with $v$ as an endpoint, this new divisor has 1 fewer chip at $v$, and 1 more chip at the other endpoint of $e$.  If we fire every vertex in a subset $A \subseteq V(G)$, we say that we fire the subset $A$.  (The resulting divisor is independent of the order in which one fires the vertices in $A$.)  We say that two divisors are \emph{equivalent} if one can be obtained from the other by a sequence of chip-firing moves.  Given a vertex $v \in V(G)$, an effective divisor $D$ is \emph{$v$-reduced} if, for every subset $A \subseteq V(G) \smallsetminus \{ v \}$, the divisor obtained by firing $A$ is not effective.  Every effective divisor is equivalent to a unique $v$-reduced divisor.  We say that a divisor $D$ on $G$ has \emph{positive rank} if its $v$-reduced representative contains $v$ in its support, for every vertex $v \in V(G)$.  The \emph{gonality} of $G$ is the minimum degree of a divisor of positive rank on $G$.

\section{Brambles and Scrambles}
\label{Sec:Scramble}

We make the following definition.

\begin{definition}
A \emph{scramble} in a graph $G$ is a set $\mathscr{S} = \{ E_1 , \ldots E_n \}$ of connected subsets of $V(G)$.
\end{definition}

We will often refer to the subsets $E_i$ as \emph{eggs}.  Scrambles with certain properties have been studied extensively in the graph theory literature.

\begin{definition}
A \emph{bramble} is a scramble $\mathscr{S}$ with the property that $E \cup E'$ is connected for every pair $E, E' \in \mathscr{S}$.  It is called a \emph{strict bramble} if every pair of elements $E, E' \in \mathscr{S}$ has nonempty intersection.
\end{definition}

\begin{definition}
A set $C \subseteq V(G)$ is called a \emph{hitting set} for a scramble $\mathscr{S}$ if $C \cap E \neq \emptyset$ for all $E \in \mathscr{S}$.
\end{definition}

The \emph{order} of a bramble $\mathscr{B}$ is the minimum size of a hitting set for $\mathscr{B}$.  The \emph{bramble number} of a graph $G$ is the maximum order of a bramble in $G$, and is denoted $\mathrm{bn}(G)$.  A result of Seymour and Thomas shows that the bramble number of a graph is closely related to another well-known graph invariant, known as the \emph{treewidth} $\mathrm{tw}(G)$.  In particular, $\mathrm{tw}(G) = \mathrm{bn}(G) - 1$ for any graph $G$ \cite{SeymourThomas}.  Here, we define some related notions for more general scrambles.

\begin{definition}
\label{Def:ScrambleOrder}
The \emph{scramble order} of a scramble $\mathscr{S}$ is the maximum integer $k$ such that:
\begin{enumerate}
\item  no set $C \subseteq V(G)$ of size less than $k$ is a hitting set for $\mathscr{S}$, and
\item  if $A \subseteq V(G)$ is a set such that there exists $E, E' \in \mathscr{S}$ with $E \subseteq A$ and $E' \subseteq A^c$, then $\vert E(A,A^c) \vert \geq k$.
\end{enumerate}
The scramble order of a scramble $\mathscr{S}$ is denoted $\vert \vert \mathscr{S} \vert \vert$.
The \emph{scramble number} of a graph $G$, denoted $\mathrm{sn}(G)$, is the maximum scramble order of a scramble in $G$.
\end{definition}

We note the following observations about the scramble order of brambles.

\begin{lemma}
The order of a strict bramble is equal to its scramble order.
\end{lemma}

\begin{proof}
Let $\mathscr{B}$ be a strict bramble of order $k$.  By definition, there is a hitting set $C \subseteq V(G)$ of size $k$ for $\mathscr{B}$, and no such set of size less than $k$.  The scramble order of $\mathscr{B}$ is therefore at most $k$.  Since $\mathscr{B}$ is a strict bramble, any two sets $E, E' \in \mathscr{B}$ have nonempty intersection.  It follows that there is no set $A \subseteq V(G)$ that contains $E$ and whose complement contains $E'$, so property (2) of Definition~\ref{Def:ScrambleOrder} is satisfied vacuously.
\end{proof}

\begin{lemma}
\label{Lem:BrambleScramble}
Let $\mathscr{B}$ be a bramble of order $k$.  Then the scramble order of $\mathscr{B}$ is either $k$ or $k-1$.
\end{lemma}

\begin{proof}
By definition, there is a hitting set $C \subseteq V(G)$ of size $k$ for $\mathscr{B}$, and no such set of size less than $k$.  The scramble order of $\mathscr{B}$ is therefore at most $k$.  By \cite[Lemma~3.3]{dbG}, if $E, E' \in \mathscr{B}$ and $A \subseteq V(G)$ is a subset such that $E \subseteq A$ and $E' \subseteq A^c$, then $\vert E(A,A^c) \vert \geq k-1$.  It follows that the scramble order of $\mathscr{B}$ is at least $k-1$. 
\end{proof}

\begin{corollary}
\label{Cor:TW}
For any graph $G$, we have $\mathrm{tw}(G) \leq \mathrm{sn}(G)$.
\end{corollary}

\begin{proof}
Let $\mathscr{B}$ be a bramble of maximum order $k$ in $G$.  By \cite{SeymourThomas}, we have $\mathrm{tw}(G) = k-1$.  By Lemma~\ref{Lem:BrambleScramble}, the scramble order of $\mathscr{B}$ is at least $k-1$, hence $\mathrm{sn}(G) \geq k-1$.
\end{proof}

\section{Properties of the Scramble Number}

We now prove our main result about the scramble number.  Namely, that the scramble number of a graph is a lower bound for the graph's gonality.  Our argument follows closely that of \cite[Theorem~3.1]{dbG}, which shows that the treewidth of a graph is a lower bound for the graph's gonality.  Indeed, we defined the scramble number with the specific goal of stating \cite[Theorem~3.1]{dbG} in its maximum generality.

\begin{theorem}
\label{Thm:Bound}
For any graph $G$, we have $\mathrm{sn}(G) \leq \mathrm{gon}(G)$.
\end{theorem}

\begin{proof}
Let $\mathscr{S}$ be a scramble on $G$, and let $D'$ be a divisor of positive rank on $G$.  We will show that $\mathrm{deg}(D') \geq \vert \vert \mathscr{S} \vert \vert$.  Among the effective divisors equivalent to $D'$, we choose $D$ such that $\mathrm{Supp}(D)$ intersects a maximum number of eggs in $\mathscr{S}$.  If $\mathrm{Supp}(D)$ is a hitting set for $\mathscr{S}$ then, by definition,
\[
\mathrm{deg}(D) \geq \vert \mathrm{Supp}(D) \vert \geq \vert \vert \mathscr{S} \vert \vert .
\]

Conversely, suppose that there is some egg $E \in \mathcal{S}$ that does not intersect $\mathrm{Supp}(D)$, and let $v\in E$.  Since $D$ has positive rank and $v \notin \mathrm{Supp}(D)$, it follows that $D$ is not $v-$reduced.  Therefore there exists a chain
\[
\emptyset \subsetneq U_1 \subseteq \cdots \subseteq U_k \subseteq V(G) \smallsetminus \{v \}
\]
and a sequence of effective divisors $D_0 , D_1, \ldots, D_k$ such that:
\begin{enumerate}
\item  $D_0 = D$,
\item  $D_k$ is $v$-reduced, and
\item  $D_i$ is obtained from $D_{i-1}$ by firing the set $U_i$, for all $i$.
\end{enumerate} 
Since $D$ has positive rank, we see that $v \in \mathrm{Supp}(D_k)$ and hence $\mathrm{Supp}(D_k)$ intersects $E$.  By assumption, $\mathrm{Supp}(D_k)$ does not intersect more eggs than $\mathrm{Supp}(D)$, so there is at least one egg $E'$ that intersects $\mathrm{Supp}(D)$ but not $\mathrm{Supp}(D_k)$.  Let $i \leq k$ be the smallest index such that there is some $E' \in \mathscr{S}$ that intersects $\mathrm{Supp}(D)$ but not $\mathrm{Supp}(D_i)$.  Then $E' \cap \mathrm{Supp}(D_{i-1}) \neq \emptyset$ and $E' \cap \mathrm{Supp}(D_i) = \emptyset$.  By \cite[Lemma~3.2]{dbG}, it follows that $E' \subseteq U_i$.

Again, by assumption, $\mathrm{Supp}(D_{i-1})$ does not intersect more eggs than $\mathrm{Supp}(D)$, so $\mathrm{Supp}(D_{i-1})$ does not intersect $E$.  Let $j \geq i$ be the smallest index such that $E \cap \mathrm{Supp}(D_{j-1}) = \emptyset$ and $E \cap \mathrm{Supp}(D_j) \neq \emptyset$.  Since $D_{j-1}$ can be obtained from $D_j$ by firing $U_j^c$, we see that $E \subseteq U_j^c \subseteq U_i^c$.  Since $E \subseteq U_i^c$ and $E' \subseteq U_i$, it follows by the definition of a scramble that $\vert E(U_i , U_i^c) \vert \geq \vert \vert \mathscr{S} \vert \vert$.  Since 
\[
\mathrm{deg}(D_{i-1}) \geq \sum_{u \in U_i} D_{i-1}(u) \geq \vert E(U_i, U_i^c) \vert,
\]
we have
\[
\mathrm{deg}(D_{i-1}) \geq \vert \vert \mathscr{S} \vert \vert .
\]
\end{proof}

We include here some observations about graphs of low scramble number.

\begin{corollary}
\label{Cor:Tree}
The scramble number of a graph $G$ is 1 if and only if $G$ is a tree.
\end{corollary}

\begin{proof}
If $G$ is a tree, then
\[
1 = \mathrm{tw}(G) \leq \mathrm{sn}(G) \leq \mathrm{gon}(G) = 1,
\]
so $\mathrm{sn}(G) = 1$.  On the other hand, if $\mathrm{sn}(G)=1$, then $\mathrm{tw}(G) \leq \mathrm{sn} (G) = 1$.  If $G$ is a simple graph, then this implies that $G$ is a tree.  To see that $G$ must be simple, let $v$ and $w$ be two adjacent vertices in $G$.  If there are multiple edges between $v$ and $w$, then the scramble $\mathscr{S} = \{ \{ v \} , \{ w \} \}$ has scramble order 2.  Since $\mathrm{sn} (G) = 1$, it follows that any pair of vertices in $G$ is connected by at most 1 edge.
\end{proof}

\begin{corollary}
\label{Lem:Cycle}
If $G$ is a cycle, then $\mathrm{sn}(G) = 2$.
\end{corollary}

\begin{proof}
If $G$ is a cycle, then
\[
2 = \mathrm{tw}(G) \leq \mathrm{sn}(G) \leq \mathrm{gon}(G) = 2,
\]
so $\mathrm{sn}(G) = 2$.
\end{proof}

One of the major advantages of the treewidth bound from \cite{dbG} is that the treewidth is minor monotone.  In other words, if $G'$ is a graph minor of a graph $G$, then $\mathrm{tw}(G') \leq \mathrm{tw}(G)$.  This is not true for the scramble number, as the following example shows.

\begin{example}
\label{Ex:Minor}
Let $G$ be the graph depicted in Figure~\ref{Fig:Gon3}.  If $v$ is the green vertex, then the divisor $3v$ has positive rank.  It follows that the gonality of $G$ is at most 3, and thus the scramble number of $G$ is at most 3 by Theorem~\ref{Thm:Bound}.

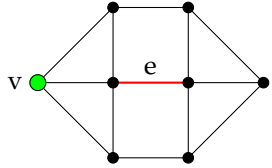
\begin{figure}[h]

\begin{tikzpicture}
\draw (0,1)--(1,0);
\draw (0,1)--(1,1);
\draw (0,1)--(1,2);
\draw (1,0)--(1,1);
\draw (1,0)--(2,0);
\draw (1,1)--(1,2);
\draw[red][thick] (1,1)--(2,1);
\draw (1.5,1.2) node{e};
\draw (1,2)--(2,2);
\draw (2,0)--(2,1);
\draw (2,0)--(3,1);
\draw (2,1)--(2,2);
\draw (2,1)--(3,1);
\draw (2,2)--(3,1);
\filldraw[green] (0,1) circle (3pt);
\draw (-0.3,1) node{v};
\draw (0,1) circle (3pt);
\filldraw (1,0) circle (2pt);
\filldraw (1,1) circle (2pt);
\filldraw (1,2) circle (2pt);
\filldraw (2,0) circle (2pt);
\filldraw (2,1) circle (2pt);
\filldraw (2,2) circle (2pt);
\filldraw (3,1) circle (2pt);
\end{tikzpicture}
\caption{A graph $G$ with scramble number 3.}
\label{Fig:Gon3}
\end{figure}

Now, let $G'$ be the graph pictured in Figure~\ref{Fig:S4}, obtained by contracting the red edge $e$ in $G$.  The 4 colored subsets are the eggs of a scramble $\mathscr{S}$, which we now show has scramble order 4.  Because the 4 eggs are disjoint, there is no hitting set of size less than 4.  Now, let $A \subseteq V(G')$ be a set with the property that both $A$ and $A^c$ contain an egg.  By exchanging the roles of $A$ and $A^c$, we may assume that $A$ contains the center red vertex.  If $A$ consists solely of this vertex, then $\vert E(A,A^c) \vert = 6$.  Otherwise, $A$ contains some, but not all, of the vertices on the hexagonal outer ring.  We then see that $E(A,A^c)$ contains at least two edges in the hexagonal outer ring, and at least two edges that have the center red vertex as an endpoint.  Thus, $\vert E(A,A^c) \vert \geq 4$.

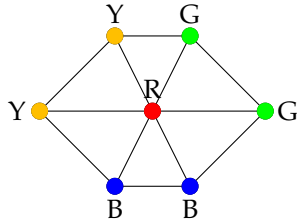
\begin{figure}[h]
\begin{tikzpicture}
\draw (0,1) circle (3pt);
\draw (1,0) circle (3pt);
\draw (1.5,1) circle (3pt);
\draw (1,2) circle (3pt);
\draw (2,0) circle (3pt);
\draw (3,1) circle (3pt);
\draw (2,2) circle (3pt);
\draw (0,1)--(1,0);
\draw (0,1)--(3,1);
\draw (0,1)--(1,2);
\draw (1,0)--(2,2);
\draw (1,0)--(2,0);
\draw (1,2)--(2,0);
\draw (1,2)--(2,2);
\draw (2,2)--(3,1);
\draw (2,0)--(3,1);
\filldraw[gold] (0,1) circle (3pt);
\draw (-0.3,1) node{Y};
\filldraw[blue] (1,0) circle (3pt);
\draw (1,-0.3) node{B};
\filldraw[red] (1.5,1) circle (3pt);
\draw (1.5,1.3) node{R};
\filldraw[gold] (1,2) circle (3pt);
\draw (1,2.3) node{Y};
\filldraw[blue] (2,0) circle (3pt);
\draw (2,-0.3) node{B};
\filldraw[green] (3,1) circle (3pt);
\draw (3.3,1) node{G};
\filldraw[green] (2,2) circle (3pt);
\draw (2,2.3) node{G};
\end{tikzpicture}

\caption{The graph $G'$ is a graph minor of $G$ with higher scramble number.}
\label{Fig:S4}

\end{figure}

\end{example}

While the scramble number is not minor monotone, it is subgraph monotone.

\begin{proposition}
\label{Prop:Subgraph}
If $G'$ is a subgraph of $G$, then $\mathrm{sn}(G') \leq \mathrm{sn}(G)$.
\end{proposition}

\begin{proof}
Let $\mathscr{S}'$ be a scramble on $G'$, and let $\mathscr{S}$ be the scramble on $G$ with the same eggs as $\mathscr{S}'$ on $G'$.  We will show that $\vert \vert \mathscr{S} \vert \vert \geq \vert \vert \mathscr{S}' \vert \vert$.  If $C \subseteq V(G)$ is a hitting set for $\mathscr{S}$, then $C \cap V(G')$ is a hitting set for $\mathscr{S}'$.  Thus, for every hitting set $C$ of $\mathscr{S}$, we have $\vert C \vert \geq \vert \vert \mathscr{S}' \vert \vert$.  Now, let $A$ be a subset of $V(G)$ such that $A$ and $A^c$ both contain eggs of $\mathscr{S}$.  Then $A \cap V(G')$ is a subset of $V(G')$ with the property that both it and its complement contain eggs of $\mathscr{S}'$, and $\vert E_G (A,A^c) \vert \geq \vert E_{G'}(A \cap V(G'), A^c \cap V(G')) \vert$.
It follows that $\vert \vert \mathscr{S} \vert \vert \geq \vert \vert \mathscr{S}' \vert \vert$.
\end{proof}

The scramble number is also invariant under subdivision.

\begin{proposition}
\label{Prop:Refinement}
If $G'$ is a subdivision of $G$, then $\mathrm{sn} (G) = \mathrm{sn} (G')$.
\end{proposition}

\begin{proof}
By induction, it suffices to consider the case where $G$ has one fewer vertex than $G'$.  Let $v$ and $w$ be adjacent vertices in $G$, and let $G'$ be the graph obtained by subdividing an edge between $v$ and $w$, introducing a vertex $u$ between them.

First, we will show that $\mathrm{sn} (G) \leq \mathrm{sn} (G')$.  To see this, let $\mathscr{S}$ be a scramble on $G$.  For each egg $E \in \mathscr{S}$, we define a connected subset $E' \subseteq V(G')$ as follows.  If $v \notin E$, then $E' = E$, and if $v \in E$, then $E' = E \cup \{ u \}$.  Let
\[
\mathscr{S}' = \{ E' \vert E \in \mathscr{S} \} .
\]
We will show that $\vert \vert \mathscr{S}' \vert \vert \geq \vert \vert \mathscr{S} \vert \vert$.

Let $C \subseteq V(G')$ be a hitting set for $\mathscr{S}'$.  If $u \notin C$, then $C$ is also a hitting set for $\mathscr{S}$.  On the other hand, if $u \in C$, then since every egg in $\mathscr{S}'$ that contains $u$ also contains $v$, the set $C' = C \cup \{ v \} \smallsetminus \{ u \}$ is a hitting set for $\mathscr{S}'$ with the property that $u \notin C'$ and $\vert C' \vert \leq \vert C \vert$.  Now, let $A$ be a subset of $V(G')$ such that both $A$ and $A^c$ contain eggs of $\mathscr{S}'$.  By exchanging $A$ and $A^c$, we may assume that $u \notin A$.  We may then think of $A$ also as a subset of $V(G)$ with the property that both $A$ and $A^c$ contain eggs of $\mathscr{S}$.  If both $v$ and $w$ are contained in $A$, then the number of edges leaving $A$ in $V(G)$ is 2 fewer than the number of edges leaving $A$ in $V(G')$.  Otherwise, these two numbers are equal.  It follows that $\vert \vert \mathscr{S}' \vert \vert \geq \vert \vert \mathscr{S} \vert \vert$.

We now show that $\mathrm{sn} (G) \geq \mathrm{sn} (G')$.  To see this, let $\mathscr{S}'$ be a scramble on $G'$ of maximal scramble order.  If $\mathrm{sn} (G) = 1$, then by Corollary~\ref{Cor:Tree}, we see that $G$ is a tree.  It follows that $G'$ is a tree as well, and $\mathrm{sn}(G') = 1$ by another application of Corollary~\ref{Cor:Tree}.  We may therefore assume that $\mathrm{sn} (G) \geq 2$.  If $\vert \vert \mathscr{S}' \vert \vert \leq 2$ the result follows, so we assume from here on that $\vert \vert \mathscr{S}' \vert \vert \geq 3$.

If every egg in $\mathscr{S}'$ contains $u$, then $\mathscr{S}'$ has a hitting set of size 1, a contradiction.  It follows that if $\{ u \} \in \mathscr{S}'$, then the set $A = \{ u \}$ has the property that both $A$ and $A^c$ contain eggs of $\mathscr{S}'$.  Thus, $\vert \vert \mathscr{S}' \vert \vert \leq \vert E_{G'} (A,A^c) \vert = 2$, another contradiction.  We may therefore assume that $\{ u \} \notin \mathscr{S}'$.  Let 
\[
\mathscr{S} = \{ E' \cap V(G) \vert E' \in \mathscr{S}' \}.
\]
Note that each element of $\mathscr{S}$ is a connected subset of $V(G)$, so $\mathscr{S}$ is a scramble in $G$.  We will show that $\vert \vert \mathscr{S} \vert \vert \geq \vert \vert \mathscr{S}' \vert \vert$.  First, let $C \subseteq V(G)$ be a hitting set for $\mathscr{S}$.  Since $\{ u \} \notin \mathscr{S}'$, we see that $C$ is also a hitting set for $\mathscr{S}'$, so $\vert C \vert \geq \vert \vert \mathscr{S}' \vert \vert$.

Now, let $A$ be a subset of $V(G)$ with the property that both $A$ and $A^c$ contain eggs of $\mathscr{S}$.  Without loss of generality, assume that $v \in A$.  By considering several cases, we will show that $\vert \vert \mathscr{S}' \vert \vert \leq \vert E_G (A,A^c) \vert$.  First, assume that there is an egg $E' \in \mathscr{S}'$ such that $E' \cap V(G) \subseteq A^c$ and $u \notin E'$.  In this case, let $A' = A \cup \{ u \}$.  We see that $A'$ contains an egg in $\mathscr{S}'$, $A'^c$ contains the egg $E'$, and $\vert E_G (A,A^c) \vert = \vert E_{G'}(A',A'^c) \vert$.  It follows that $\vert \vert \mathscr{S}' \vert \vert \leq \vert E_G (A,A^c) \vert$.  Conversely, if there is an egg $E' \in \mathscr{S}'$ such that $E' \cap V(G) \subseteq A^c$ and $u \in E'$, then since eggs are connected and $E' \neq \{ u \}$, we have $w \in E' \cap V(G) \subseteq A^c$.  In this case, if there is an egg $E' \in \mathscr{S}'$ such that $E' \cap V(G) \subseteq A$ and $u \notin E'$, then by a similar argument setting $A' = A$, we see that $\vert \vert \mathscr{S}' \vert \vert \leq \vert E_G (A,A^c) \vert$.

Finally, consider the case where every egg $E' \in \mathscr{S}'$ such that $E' \cap V(G)$ is contained in either $A$ or $A^c$ contains $u$.  As above, this implies that the edge between $v$ and $w$ must be in $E_G (A,A^c)$.  We will construct a hitting set $C$ for $\mathscr{S}'$ of size $\vert C \vert \leq \vert E_G (A,A^c) \vert$, thus showing that $\vert \vert \mathscr{S}' \vert \vert \leq \vert E_G (A,A^c) \vert$.  To construct $C$, let $u \in C$, and, for every edge in $E_G (A,A^c)$ other than the edge between $v$ and $w$, choose one of its endpoints to be in $C$.  Clearly, $\vert C \vert \leq \vert E_G (A,A^c) \vert$.  To see that $C$ is a hitting set, let $E' \in \mathscr{S}'$ be an egg.  If $E' \cap V(G)$ is contained in either $A$ or $A^c$, then $u \in E' \cap C$ by assumption.  On the other hand, if $E' \cap V(G)$ is contained in neither $A$ nor $A^c$, then since $E'$ is connected, the set $E_{G'}(E' \cap A, E' \cap A^c)$ is nonempty.   Since $C$ contains an endpoint of every edge in this set, it follows that $E' \cap C$ is nonempty.
\end{proof}

\begin{example}
\label{Ex:Hyp}
The graph on the left in Figure~\ref{Fig:Refinement} has gonality 2.  By Theorem~\ref{Thm:Bound}, its scramble number is at most 2.  Since it is not a tree, by Corollary~\ref{Cor:Tree}, its scramble number is exactly 2.

On the other hand, the graph on the right has gonality 3.  Since it is a subdivision of the graph on the left, however, by Proposition~\ref{Prop:Refinement} the two graphs have the same scramble number.  Thus, the graph on the right is an example where the gonality and scramble number disagree.

\begin{figure}[h]

\begin{tikzpicture}
\filldraw (0,1) circle (2pt);
\filldraw (1,0) circle (2pt);
\filldraw (1,2) circle (2pt);
\filldraw (2,0) circle (2pt);
\filldraw (2,2) circle (2pt);
\filldraw (3,1) circle (2pt);
\draw (0,1)--(1,0);
\draw (0,1)--(1,2);
\draw (1,0)--(1,1);
\draw (1,0)--(2,0);
\draw (1,1)--(1,2);
\draw (1,2)--(2,2);
\draw (2,0)--(2,1);
\draw (2,0)--(3,1);
\draw (2,1)--(2,2);
\draw (2,2)--(3,1);

\filldraw (4,1) circle (2pt);
\filldraw (5,0) circle (2pt);
\filldraw (5,2) circle (2pt);
\filldraw (6,0) circle (2pt);
\filldraw (6,2) circle (2pt);
\filldraw (7,1) circle (2pt);
\filldraw (5.5,0) circle (2pt);
\draw (4,1)--(5,0);
\draw (4,1)--(5,2);
\draw (5,0)--(5,1);
\draw (5,0)--(6,0);
\draw (5,1)--(5,2);
\draw (5,2)--(6,2);
\draw (6,0)--(6,1);
\draw (6,0)--(7,1);
\draw (6,1)--(6,2);
\draw (6,2)--(7,1);
\end{tikzpicture}
\caption{Two graphs with the same scramble number, but different gonalities}
\label{Fig:Refinement}
\end{figure}
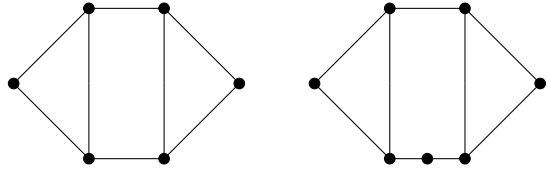
\end{example}

The following examples, suggested by the anonymous referee, show that $\mathrm{sn}(G)$ is unbounded in $\mathrm{tw}(G)$, and $\mathrm{gon}(G)$ is unbounded in $\mathrm{sn}(G)$.

\begin{example}
\label{Ex:Path}
In this example, we construct a family of graphs, all with treewidth 2, that have increasingly large scramble number.
For $k \geq 2$, let $G_k$ be the graph obtained from the path graph $P_k$ on $k$ vertices by replacing every edge with $k$ parallel edges.  For example, $G_3$ is pictured in Figure~\ref{Fig:SP}.  The treewidth of a graph with parallel edges is equal to that of the underlying simple graph.  Because $G_k$ is obtained from a path graph by introducing parallel edges, we have $\mathrm{tw} (G_k) = 1$.  On the other hand, we will see that $\mathrm{sn} (G_k) = k$.  To see this, note that $\mathrm{gon} (G_k) = k$, so it suffices to construct a scramble $\mathscr{S}$ of order $k$.  Let $\mathscr{S}$ be the scramble whose eggs are the individual vertices of $G_k$.  Since the eggs are disjoint, a minimal hitting set has size $k$.  If $A \subsetneq V(G_k)$ is a non-empty subset, then $\vert E(A,A^c) \vert \geq k$.  It follows that $\vert \vert \mathscr{S} \vert \vert = k$.
\begin{figure}[h]
\scalebox{1.5}{
\begin{tikzpicture}
\draw (0,0) circle (0.5);
\draw (1,0) circle (0.5);
\draw (-0.5,0)--(0.5,0);
\draw (0.5,0)--(1.5,0);
\draw [ball color=black] (-0.5,0) circle (0.55mm);
\draw [ball color=black] (0.5,0) circle (0.55mm);
\draw [ball color=black] (1.5,0) circle (0.55mm);
\end{tikzpicture}
}
\caption{The graph $G_3$.}
\label{Fig:SP}
\end{figure}
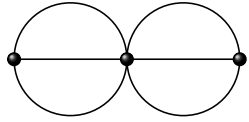

\end{example}

\begin{example}
\label{Ex:COL}
In this example we construct a family of graphs, all with scramble number 2, that have increasingly large gonality.
Let $G_k$ be the graph obtained from the path graph $P_k$ on $k$ vertices by replacing every edge with 2 parallel edges.  Since $G_k$ is not a tree, we have
\[
2 \leq \mathrm{sn} (G_k) \leq \mathrm{gon} (G_k) = 2,
\]
hence $\mathrm{sn} (G_k) = 2$.  Now, let $G'_k$ be the graph obtained from $G_k$ by subdividing one edge from each pair $k-1$ times.  An example of such a graph appears in Figure~\ref{Fig:COL}.  The graph $G'_k$ is a chain of $k-1$ loops, suitably chosen to be Brill-Noether general by \cite{CDPR}.  In particular, $\mathrm{gon}(G'_k) = \lceil \frac{k+2}{2} \rceil$.  By Proposition~\ref{Prop:Refinement}, however, we have $\mathrm{sn} (G'_k) = \mathrm{sn} (G_k) = 2$.  

\begin{figure}[h]
\scalebox{1.5}{
\begin{tikzpicture}
\draw (0,0) circle (0.5);
\draw (1,0) circle (0.5);
\draw (2,0) circle (0.5);
\draw [ball color=black] (-0.5,0) circle (0.55mm);
\draw [ball color=black] (-0.35,0.35) circle (0.55mm);
\draw [ball color=black] (0,0.5) circle (0.55mm);
\draw [ball color=black] (0.35,0.35) circle (0.55mm);
\draw [ball color=black] (0.5,0) circle (0.55mm);
\draw [ball color=black] (0.65,0.35) circle (0.55mm);
\draw [ball color=black] (1,0.5) circle (0.55mm);
\draw [ball color=black] (1.35,0.35) circle (0.55mm);
\draw [ball color=black] (1.5,0) circle (0.55mm);
\draw [ball color=black] (1.65,0.35) circle (0.55mm);
\draw [ball color=black] (2,0.5) circle (0.55mm);
\draw [ball color=black] (2.35,0.35) circle (0.55mm);
\draw [ball color=black] (2.5,0) circle (0.55mm);
\end{tikzpicture}
}
\caption{The graph $G'_4$.}
\label{Fig:COL}
\end{figure}
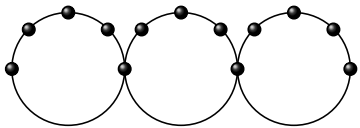
\end{example}

\section{Examples}
\label{Sec:Examples}

In this section, we compute the scramble numbers and gonalities of several well-known families of graphs.  Our hope is that these examples illustrate the advantages of the scramble number as a tool for computing gonality, as our constructions are relatively simple in comparison to the preexisting literature.

Our examples all arise as Cartesian products of graphs.  Recall that the Cartesian product of two graphs $G_1$ and $G_2$, denoted $G_1 \mathbin{\square} G_2$, is the graph with vertex set $V(G_1) \times V(G_2)$ and an edge between $(u_1, u_2)$ and $(v_1, v_2)$ if either $u_1 = v_1$ and there is an edge between $u_2$ and $v_2$, or $u_2 = v_2$ and there is an edge between $u_1$ and $v_1$.  For a fixed vertex $v \in G_1$, we refer to the set
\[
C_v = \Big\{ (v,w) \in V(G_1 \mathbin{\square} G_2) \vert w \in G_2 \Big\}
\]
as a \emph{column}.  Similarly, for $w \in G_2$, we refer to the set
\[
R_w = \Big\{ (v,w) \in V(G_1 \mathbin{\square} G_2) \vert v \in G_1 \Big\}
\]
as a \emph{row}.  A bound on the gonality of Cartesian products can be found in \cite{AidunMorrison}.

\begin{proposition}\cite[Proposition~3]{AidunMorrison}
\label{Prop:ProductBound}
For any two graphs $G_1$ and $G_2$,
\[
\mathrm{gon}(G_1 \mathbin{\square} G_2) \leq \min \Big\{ \mathrm{gon}(G_1) \vert V(G_2) \vert, \mathrm{gon}(G_2) \vert V(G_1) \vert \Big\} .
\]
\end{proposition}

We provide several examples where this bound is achieved.  It is a standard result that the $m \times n$ grid graph has treewidth $\min \{ m,n \}$, and it is shown in \cite{dbG} that such graphs have gonality $\min \{ m,n \}$ as well.  A grid graph is an example of the product of two trees, a family of graphs whose gonality is computed in \cite{AidunMorrison}.  We reproduce this result here using the scramble number.

\begin{proposition}\cite[Proposition~11]{AidunMorrison}
If $T_1$ and $T_2$ are trees, then
\[
\mathrm{gon} (T_1 \mathbin{\square} T_2) = \mathrm{sn} (T_1 \mathbin{\square} T_2) = \min \Big\{ \vert V(T_1) \vert, \vert V(T_2) \vert \Big\} .
\]
\end{proposition}

\begin{proof}
By Proposition~\ref{Prop:ProductBound}, the gonality of $T_1 \mathbin{\square} T_2$ is at most $\min \{ \vert V(T_1) \vert, \vert V(T_2) \vert \}$.  We therefore seek to bound the gonality from below.  By Theorem~\ref{Thm:Bound}, it suffices to construct a scramble of scramble order $\min \{ \vert V(T_1) \vert, \vert V(T_2) \vert \}$.

Let $\mathscr{S}$ be the set of columns in $T_1 \mathbin{\square} T_2$.  Any row $R_w$ is a hitting set for $\mathscr{S}$, and $\vert R_w \vert = \vert V(T_1) \vert$.  Moreover, if $v \in T_1$ is a leaf, then $\vert E( C_v , C_v^c ) \vert = \vert V(T_2) \vert$.  It follows that
\[
\vert \vert \mathscr{S} \vert \vert \leq \min \Big\{ \vert V(T_1) \vert, \vert V(T_2) \vert \Big\} .
\]

Since the number of columns is $\vert V(T_1) \vert$ and they are disjoint, there is no hitting set of size less than $\vert V(T_1) \vert$.  Now, let $A$ be a subset of $V(T_1 \mathbin{\square} T_2)$ with the property that both $A$ and $A^c$ contain a column.  Then every row of $T_1 \mathbin{\square} T_2$ contains a vertex in $A$ and a vertex in $A^c$, so every row contains an edge in $E(A,A^c)$.  It follows that $\vert E(A,A^c) \vert$ is greater than or equal to the number of rows, which is $\vert V(T_2) \vert$, hence
\[
\vert \vert \mathscr{S} \vert \vert \geq \min \Big\{ \vert V(T_1) \vert, \vert V(T_2) \vert \Big\} .
\]
\end{proof}

In \cite{Morrison}, the authors compute the treewidth of the \emph{stacked prism graphs} $Y_{m,n}$, the product of a cycle with $m$ vertices and a path with $n$ vertices.  They show that the gonality of $Y_{m,n}$ is equal to its treewdith, except in the special case where $m = 2n$.  We prove the following generalization, which holds even in this special case.  Even in the cases where the gonality has been previously computed, we believe that our constructions, using scrambles rather than brambles, are much simpler.  For this reason, we have treated these graphs for all $m$ and $n$ uniformly.

\begin{proposition}
\label{Prop:StackedPrism}
If $C$ is a cycle and $T$ is a tree, then
\[
\mathrm{gon} (C \mathbin{\square} T) = \mathrm{sn} (C \mathbin{\square} T) = \min \Big\{ \vert V(C) \vert, 2 \vert V(T) \vert \Big\} .
\]
\end{proposition}

\begin{proof}
By Proposition~\ref{Prop:ProductBound}, we have $\mathrm{gon} (C \mathbin{\square} T) \leq \min \{ \vert V(C) \vert, 2\vert V(T) \vert \}$.  We now compute a lower bound.  By Theorem~\ref{Thm:Bound}, it suffices to construct a scramble of scramble order $\min \{ \vert V(C) \vert, 2 \vert V(T) \vert \}$.

Again, we let $\mathscr{S}$ be the set of columns in $C \mathbin{\square} T$.  (See, for example, Figure~\ref{Fig:StackedPrism}.)  Any row $R_w$ is a hitting set for $\mathscr{S}$, and $\vert R_w \vert = \vert V(C) \vert$.  Moreover, for any $v \in C$ we have $\vert E( C_v , C_v^c ) \vert = 2 \vert V(T) \vert$.  It follows that
\[
\vert \vert \mathscr{S} \vert \vert \leq \min \Big\{ \vert V(C) \vert, 2 \vert V(T) \vert \Big\} .
\]

Since the number of columns is $\vert V(C) \vert$ and they are disjoint, there is no hitting set of size less than $\vert V(C) \vert$.  Now, let $A$ be a subset of $V(C \mathbin{\square} T)$ with the property that both $A$ and $A^c$ contain a column.  Then every row of $C \mathbin{\square} T$ contains a vertex in $A$ and a vertex in $A^c$, so every row contains at least two edges in $E(A,A^c)$.  It follows that $\vert E(A,A^c) \vert$ is greater than or equal to twice the number of rows, which is $\vert V(T) \vert$, hence
\[
\vert \vert \mathscr{S} \vert \vert \geq \min \Big\{ \vert V(C) \vert, 2 \vert V(T) \vert \Big\} .
\]
\end{proof}

\begin{figure}[h]

\begin{tikzpicture}
\draw (0,1) circle (3pt);
\draw (1,0) circle (3pt);
\draw (1,1) circle (3pt);
\draw (0,0) circle (3pt);
\draw (2,2) circle (3pt);
\draw (2,-1) circle (3pt);
\draw (-1,2) circle (3pt);
\draw (-1,-1) circle (3pt);
\draw (0,0)--(1,0);
\draw (0,0)--(0,1);
\draw (0,1)--(1,1);
\draw (1,0)--(1,1);
\draw (2,2)--(2,-1);
\draw (2,2)--(-1,2);
\draw (-1,-1)--(2,-1);
\draw (-1,-1)--(-1,2);
\draw (-1,-1)--(0,0);
\draw (2,-1)--(1,0);
\draw (-1,2)--(0,1);
\draw (2,2)--(1,1);
\filldraw[green] (0,1) circle (3pt);
\draw (0,1.3) node{G};
\filldraw[red] (1,0) circle (3pt);
\draw (1,-0.3) node{R};
\filldraw[blue] (1,1) circle (3pt);
\draw (1,1.3) node{B};
\filldraw[yellow] (0,0) circle (3pt);
\draw (0,-0.3) node{Y};
\filldraw[blue] (2,2) circle (3pt);
\draw (2,2.3) node{B};
\filldraw[red] (2,-1) circle (3pt);
\draw (2,-1.3) node{R};
\filldraw[green] (-1,2) circle (3pt);
\draw (-1,2.3) node{G};
\filldraw[yellow] (-1,-1) circle (3pt);
\draw (-1,-1.3) node{Y};
\end{tikzpicture}
\caption{The stacked prism graph $Y_{4,2}$ with a scramble of scramble order 4.  Note that, by \cite[Proposition~3.3]{Morrison}, the treewidth of $Y_{4,2}$ is only 3.}
\label{Fig:StackedPrism}
\end{figure}
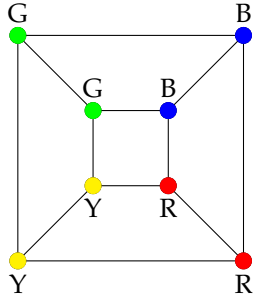

Note that in the special case where $m = 2n$, Proposition~\ref{Prop:StackedPrism} shows that the scramble number of the stacked prism graph $Y_{m,n}$ can be strictly greater than the treewidth.  In \cite{Morrison}, the authors also compute the treewidth of the \emph{toroidal grid graphs} $T_{m,n}$, the product of a cycle with $m$ vertices and a cycle with $n$ vertices.  They further show that the gonality of $T_{m,n}$ is equal to its treewidth, except in the special cases where $m = n $ or $m = n \pm 1$.  As with the stacked prism graphs, we compute the gonality of these graphs for all $m$ and $n$ uniformly, including the cases not covered in \cite{Morrison}.

\begin{proposition}
For all $m,n \geq 2$, we have
\[
\mathrm{gon}(T_{m,n}) = \mathrm{sn}(T_{m,n}) = \min \{ 2m, 2n \} .
\]
\end{proposition}

\begin{proof}
By Proposition~\ref{Prop:ProductBound}, $\mathrm{gon}(T_{m,n}) \leq \min \{ 2m, 2n \}$, so we will compute a lower bound.  By Theorem~\ref{Thm:Bound}, it suffices to construct a scramble of scramble order $\min \{ 2m, 2n \}$.  Let
\[
\mathscr{S} = \Big\{ C_v \smallsetminus \{ (v,w) \} \mid (v,w) \in V(T_{m,n}) \Big\}
\]
be the set of columns in $T_{m,n}$ with one vertex removed.  (See, for example, Figure~\ref{Fig:Torus}.)  The union of any two rows is a hitting set for $\mathscr{S}$ of size $2m$.  Moreover, for any vertex $v$ in the cycle of length $m$, we see that both $C_v$ and $C_v^c$ contain an egg, and we have $\vert E( C_v , C_v^c ) \vert = 2n$.  It follows that
\[
\vert \vert \mathscr{S} \vert \vert \leq \min \{ 2m, 2n \} .
\]

If $C$ is a subset of the vertices of size less than $2m$, then some column contains at most 1 vertex of $C$, hence $C$ is not a hitting set for $\mathscr{S}$.  Now, let $A$ be a subset of $V(T_{m,n})$ with the property that both $A$ and $A^c$ contain eggs.  Specifically, suppose that $A$ contains every vertex in column $C_v$ except for possibly $(v,w)$, and that $A^c$ contains every vertex in column $C_{v'}$ except for possibly $(v',w')$.  Note that, if $n \geq 3$, then we must have $v \neq v'$.  If $n=2$ and $v=v'$, then the $4=2n$ edges in column $C_v$ are all contained in $E(A,A^c)$.  We now assume that $v \neq v'$, and consider the rows of $T_{m,n}$.  Note that the only row that may be contained in $A$ is $R_{w'}$, and the only row that may be contained in $A^c$ is $R_w$.  If neither $A$ nor $A^c$ contains the row $R_{w''}$, then at least two edges in $R_{w''}$ are contained in $E(A,A^c)$.  We therefore see that the number of row-edges in $E(A,A^c)$ is at least:
\[
\left\{ \begin{array}{ll}
2n-4 & \textrm{if $R_{w'} \subseteq A$ and $R_w \subseteq A^c$} \\
2n-2 & \textrm{if one of $R_{w'} \subseteq A$ or $R_w \subseteq A^c$} \\
2n & \textrm{otherwise.}
\end{array} \right.
\]

We now consider column-edges.  If $R_w \subseteq A^c$, then since $A$ contains $C_v \smallsetminus \{ (v,w) \}$, we see that the two edges in column $C_v$ with endpoints $(v,w)$ are contained in $E(A,A^c)$.  Similarly, if $R_{w'} \subseteq A$, then the two edges in column $C_{v'}$ with endpoints $(v',w')$ are contained in $E(A,A^c)$.  It follows that $\vert E(A,A^c) \vert \geq 2n$, hence
\[
\vert \vert \mathscr{S} \vert \vert \geq \min \{ 2m , 2n \} .
\]
\end{proof}

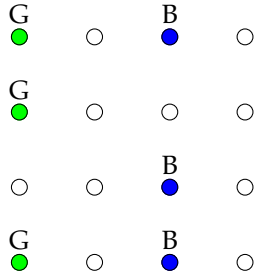
\begin{figure}[h]

\begin{tikzpicture}
\filldraw[green] (0,0) circle (3pt);
\draw (0,0.3) node{G};
\filldraw[green] (0,2) circle (3pt);
\draw (0,2.3) node{G};
\filldraw[green] (0,3) circle (3pt);
\draw (0,3.3) node{G};
\filldraw[blue] (2,0) circle (3pt);
\draw (2,0.3) node{B};
\filldraw[blue] (2,1) circle (3pt);
\draw (2,1.3) node{B};
\filldraw[blue] (2,3) circle (3pt);
\draw (2,3.3) node{B};
\draw (0,0) circle (3pt);
\draw (0,1) circle (3pt);
\draw (0,2) circle (3pt);
\draw (0,3) circle (3pt);
\draw (1,0) circle (3pt);
\draw (1,1) circle (3pt);
\draw (1,2) circle (3pt);
\draw (1,3) circle (3pt);
\draw (2,0) circle (3pt);
\draw (2,1) circle (3pt);
\draw (2,2) circle (3pt);
\draw (2,3) circle (3pt);
\draw (3,0) circle (3pt);
\draw (3,1) circle (3pt);
\draw (3,2) circle (3pt);
\draw (3,3) circle (3pt);

\end{tikzpicture}
\caption{Two representative eggs in $T_{4,4}$.}
\label{Fig:Torus}
\end{figure}

\bibliographystyle{alpha}
\bibliography{ref}

\end{document}